\pdfoutput=1
\documentclass[11pt]{article}
\usepackage{bbm}
\usepackage{amsmath}
\usepackage{amsthm}
\usepackage{amssymb} 
\usepackage{mathrsfs}
\usepackage[dvipsnames]{xcolor}
\usepackage{graphicx}
\usepackage{authblk}
\usepackage[backref=none,hypertexnames=false,colorlinks=true,urlcolor=blue,linkcolor=blue,citecolor=blue]{hyperref}
\usepackage[numbers,comma,square,sort&compress]{natbib}
\usepackage[letterpaper,text={6.5in,9in},centering]{geometry}
\usepackage{tabularx}
\usepackage{booktabs} 
\usepackage{algorithm}
\usepackage{array} 

\graphicspath{{eps/}{pdf/}}

\makeatletter\@addtoreset{equation}{section}\makeatother

\newcommand{\C}{\mathbb{C}}
\newcommand{\drm}{\mathrm{d}}

\newtheorem{thm}{Theorem}[section] 
  
\newtheorem{cor}[thm]{Corollary}

\newtheorem{exmp}{Example}[section]

\begin{document}

\title{
On the lattice property of the Koopman operator spectrum
}
\author{Jason J. Bramburger}

\affil{\small Department of Mathematics and Statistics, Concordia University, Montr\'eal, QC, Canada}

\date{}
\maketitle

\begin{abstract}
The Koopman operator has become a celebrated tool in modern dynamical systems theory for analyzing and interpreting both models and datasets. The linearity of the Koopman operator means that important characteristics about it, and in turn its associated nonlinear system, are captured by its eigenpairs and more generally its spectrum. Many studies point out that the spectrum of the Koopman operator has a multiplicative lattice structure by which eigenvalues and eigenfunctions can be multiplied to produce new eigenpairs. However, these observations fail to resolve whether the new eigenfunction remains in the domain of the Koopman operator. In this work, we prove that the spectrum of the Koopman operator associated to discrete-time dynamical systems has a multiplicative lattice structure. We further demonstrate that the Koopman operator associated to discrete-time stochastic process does not necessarily have such a structure, demonstrating an important nuance that lies at the heart of Koopman operator theory.  
\end{abstract}

\maketitle


\section{Introduction}

Since its inception nearly a century ago, the composition operator introduced by Koopman \cite{koopman1931hamiltonian} has found widespread application throughout dynamical systems to analyze complex behaviour \cite{brunton2022modern}. The now-termed Koopman operator provides an equivalent formulation of a finite-dimensional nonlinear dynamical system as a linear operator acting on an infinite-dimensional space of functions. The advantage of the linear Koopman description of the dynamical system is that its temporal evolution can be completely understood in terms of the operator's spectrum, with eigenfunctions characterizing coherent structures in the system that can only grow or decay exponentially in time according to their associated eigenvalues. This recent growth in interest around the Koopman operator has primarily been driven by computational methods that seek to approximate it using only data gathered from the system \cite{williams2015data}, as well as its connection to dynamic mode decomposition \cite{schmid2010dynamic}. The result has been the development of robust computational methods that are inspired by the Koopman operator to forecast data, identify coherent structures, and discover conservation laws; see \cite{bramburger2024data} and the references therein for a full introduction.

The general setting is as follows. Consider a discrete-time dynamical $x_{n+1} = Tx_n$ with $T:M \to M$ and $M$ a topological space. The {\em Koopman operator}, denoted $K$, acts on functions $f:M \to \C$ by 
\begin{equation}
	[Kf](x) = f(Tx),	
\end{equation} 
for each $x \in M$. It is easy to see that the Koopman operator is linear and so a function $f_\lambda$ is an eigenfunction of $K$ with eigenvalue $\lambda \in \C$ if $f_\lambda(T(x)) = \lambda f_\lambda(x)$ for all $x \in M$. Eigenfunctions are critical to understanding the dynamics of $T$ since the set $\{x\in M:\ f_\lambda(x) = 0\}$ with nonzero $\lambda$ is an invariant set of the dynamical system $T$, while eigenfunctions with $\lambda = 1$ are conserved quantities of the dynamics. Furthermore, an important property of the Koopman operator is that eigenfunctions can be multiplied to form new eigenfunctions \cite{budivsic2012applied}. Indeed, if $(\lambda,f_\lambda)$ and $(\eta,g_\eta)$ are eigenpairs of $K$, then it follows that 
\begin{equation}
    \begin{split}
	[Kf_\lambda g_\eta](x) &= f_\lambda(Tx)g_\eta(Tx) \\ 
    &= \lambda f_\lambda(x) \eta g_\eta(x) \\ 
    &= \lambda\eta f_\lambda(x)g_\eta(x), \quad \forall x \in M,
    \end{split}
\end{equation} 
showing that $(\lambda\eta,f_\lambda g_\eta)$ is an eigenpair of $K$ as well. Hence, the eigenvalues of the Koopman operator form a lattice with respect to multiplication.

The above is purely formal, since a proper discussion of the eigenvalues of a linear operator, and more generally its spectrum, requires a definition of the domain and range of the operator. A natural setting is to consider the Hilbert space 
\[
    L^2 := L^2(M,\mu) = \bigg\{f:M \to \mathbb{C}\ \bigg|\ \int_M |f(x)|^2\drm \mu(x) < \infty\bigg\}
\]
with a given probability measure $\mu$ on $M$ such that $Kf \in L^2$ for all $f \in L^2$. The dynamical system $x_{n+1} = Tx_n$ does not need to be measure-preserving with respect to $\mu$. In fact, data-driven approximations of $K$ take $\mu$ as the reference measure from which data is drawn and convergence as data points increase \cite{bramburger2024auxiliary,korda2018convergence,klus2016numerical} is provided in the $L^2$ norm 
\[
    \|f\|_2 := \sqrt{\int_M |f(x)|^2\drm \mu(x)}.
\] 
While ergodic or physical measures of $T$ are good candidates for $\mu$, other typical examples include the normalized Lebesgue measure when $M$ is compact or the Gaussian measure when $M = \mathbb{R}^d$. An important point that has drawn criticism from working in $L^2$ is that it is not a Banach algebra with respect to pointwise multiplication \cite{gonzalez2021kernel}, meaning that even if $f_\lambda,g_\eta \in L^2$ are eigenfunctions of $K$, it cannot be guaranteed that $f_\lambda g_\eta \in L^2$. Despite this drawback, $L^2$ is an optimal setting for the Koopman operator since this well-studied Hilbert space allows one to precisely define best approximations, projections, and orthonormal bases, all of which are important for data-driven techniques involving the Koopman operator. 

In this brief communication we attempt to reconcile some of these issues related to posing the Koopman operator on $L^2$. We show that the {\em spectrum} of $K:L^2 \to L^2$, as described by the set 
\begin{equation}
	\sigma(K) := \{\lambda \in \C\ |\ (K - \lambda):L^2 \to L^2\ \mathrm{has\ no\ bounded\ inverse}\},
\end{equation}
exhibits the above-described lattice property. The set of eigenvalues of $K$ belongs to the spectrum $\sigma(K)$, but since $L^2$ is often infinite-dimensional, the spectrum can be made up of more than just eigenvalues. The set of all such $\lambda \in \sigma(K)$ that are not eigenvalues is referred to as the {\em essential spectrum} of the operator $K$. The following is the characterization of the spectrum given by Weyl adopted for the present situation \cite{hislop1996essential}.

\paragraph{\bf Weyl's Criterion:} $\lambda \in \sigma(K)$ if and only if there exists a sequence $\{\psi_n\}_{n = 1}^\infty \subset L^2$ such that $\|\psi_n\|_2 = 1$ for all $n \geq 1$ and 
\begin{equation}
	\lim_{n \to \infty} \|K\psi_n - \lambda\psi_k\|_2 = 0.
\end{equation}  

Clearly, if $(\lambda,f_\lambda)$ is an eigenpair of $K$ then we may simply take $\psi_n = f_\lambda$ for all $n \geq 1$ to satisfy Weyl's criterion. When $\lambda$ is not an eigenvalue, then $\{\psi_n\}_{n = 1}^\infty$ provides a sequence of approximate eigenfunctions to $\lambda$. The following result proves the lattice structure of the Koopman operator on $L^2$. Furthermore, the proof uses Weyl's criterion to show that even if the product of eigenfunctions is not an eigenfunction itself, a sequence of approximate eigenfunctions can be constructed to converge pointwise to this product.   

\begin{thm}\label{thm:Main}
	Given a map $T: M\to M$ with $M$ a topological space, suppose its Koopman operator $K:L^2\to L^2$ is well-defined on $L^2 = L^2(M,\mu)$ for some probability measure $\mu$ on $M$. Then, if $\lambda,\eta \in \sigma(K)$, we have that $\lambda\eta \in \sigma(K)$.
\end{thm}

The proof of this result is left to Section~\ref{sec:Proof} below. First, we immediately arrive at the following corollary.

\begin{cor}
	If $K:L^2 \to L^2$ is a bounded operator, then every $\lambda \in \sigma(K)$ satisfies $|\lambda| \leq 1$.
\end{cor}

\begin{proof}
	From the fact that $K$ is assumed to be a bounded operator, it follows that $\sigma(K)$ is a compact subset of $\mathbb{C}$ and therefore bounded itself. Then, if we assume that $\lambda \in \sigma(K)$ satisfies $|\lambda| > 1$, from Theorem~\ref{thm:Main} we have that $\lambda^n \in \sigma(K)$ for all $n \geq 1$. Since $|\lambda| > 1$, it follows that $|\lambda^n|$ can be made arbitrarily large by increasing $n$, giving that $\sigma(K)$ is not bounded. This is a contradiction, thus showing no $\lambda \in \sigma(K)$ can satisfy $|\lambda| \leq 1$.
\end{proof}

\section{Lack of lattice structure for stochastic systems}\label{sec:Stochastic}

The main result of this letter proves the lattice structure of the Koopman operator associated to discrete-time deterministic dynamical systems. However, the Koopman operator can be defined for general Markovian processes in both continuous and discrete time \cite{bramburger2024auxiliary}, which begs the question as to whether one would expect a similar spectral structure for other Koopman operators. For ordinary differential equations (deterministic continuous-time processes) the Koopman operator composes scalar-valued functions with the flow operator \cite{bevanda2021koopman}, resulting in a continuous one-parameter semigroup. In this case the Koopman operator is again acting by function composition, and so it is not difficult to extend Theorem~\ref{thm:Main} to the deterministic continuous-time setting (under appropriate well-posedness assumptions on the semigroup).

The situation for stochastic dynamics appears to be more nuanced. First, if we let $x_n$ denote the state of a discrete-time stochastic process at time $n\in\mathbb{N}$ in a state space $M$, then the {\em stochastic Koopman operator} \cite{wanner2022robust} $K$ acts on scalar functions $f:M \to \mathbb{C}$ by
\begin{equation}
    Kf(x) = \mathbb{E}[f(x_{n+1})\ |\ x_n = x].    
\end{equation}
Much like for deterministic dynamics, one sees that the Koopman operator moves the system forward in time inside of the function $f$, but now this is being done in expectation. The following example provides a demonstration for which the spectrum of the stochastic Koopman operator does not have a lattice structure.

\begin{exmp}[Markov chains having no lattice structure]\normalfont 
    Consider a Markov chain on $m \geq 2$ states. The state space can be identified as $M = \{1,\dots,m\}$, while we denote the transition probabilities $p_{i,j} = \mathbb{P}(x_{n+1} = j\ |\ x_n = i)$ for each $i,j = 1,\dots,m$. In this case the Koopman operator acts on functions $f \in L^2(M,\mu) \cong \mathbb{R}^m$ by \cite{klus2024transfer}  
    \begin{equation}
        Kf(i) = \sum_{j = 1}^m p_{i,j}f(j), \quad \forall i = 1,\dots, m.
    \end{equation}
    Thus, the Koopman operator acts equivalently to the linear map $v \mapsto Pv$ for all $v \in \mathbb{R}^m$ with $P = [p_{i,j}]_{i,j = 1}^m$ the matrix of all transition probabilities. Moreover, the spectrum of the Koopman operator is comprised exclusively of the eigenvalues of the matrix $P$, of which there are only finitely many. 

    Of course, eigenvalues of matrices need not have a lattice structure. Precisely, if $P$ has even a single eigenvalue $\lambda$ for which $|\lambda| \neq 1$, then $\sigma(K)$ cannot have a lattice structure. Indeed, a lattice structure would imply that $\lambda^r$ is an eigenvalue of $P$ for every $r \geq 1$, but since $|\lambda| < 1$ we have that $\{\lambda^r\}_{r\geq 1}$ contains infinitely many distinct values. Since $P$ can only have a finite number of eigenvalues, a lattice structure is not possible here. As an example, the transition matrix 
    \begin{equation}
        P = \begin{bmatrix}
            0.9 & 0.1 \\ 0.5 & 0.5   
        \end{bmatrix}
    \end{equation}
    associated to a 2 state Markov chain gives a Koopman spectrum of $\sigma(K) = \{\frac{2}{5},1\}$, which does not have a lattice structure.
\end{exmp} 

The previous example illustrates that the spectrum of the stochastic Koopman operator does not necessarily have a lattice structure. What therefore remains as an outstanding problem for future work is identifying what stochastic processes, if any, result in a Koopman operator having such a structured spectrum.

\section{Proof of Theorem~\ref{thm:Main}}\label{sec:Proof}

In this section we will prove Theorem~\ref{thm:Main}. Suppose $\lambda,\eta \in \sigma(K)$ and let $\{f_n\}_{n = 1}^\infty$, $\{g_n\}_{n = 1}^\infty\subset L^2$ be sequences satisfying Weyl's criterion for $\lambda$ and $\eta$, respectively. By assumption the sequences $\{Kf_n - \lambda f_n\}_{n = 1}^\infty$, $\{Kg_n-\eta g_n\}_{n = 1}^\infty$ converge to zero in the $L^2$ norm, which in turn gives that they also converge to zero in measure, i.e. for all $\varepsilon > 0$ we have 
\begin{equation}
    \begin{split}
        \lim_{n\to\infty} \mu(\{x\in M\ |\ |Kf_n(x) - \lambda f_n(x)| \geq \varepsilon\}) &= 0,\\
         \lim_{n\to\infty} \mu(\{x\in M\ |\ |Kg_n(x) - \eta g_n(x)| \geq \varepsilon\}) &= 0.
    \end{split}
\end{equation} 
This convergence in measure implies that there exists a subsequence which converges pointwise $\mu$-almost everywhere, and since $\mu$ is assumed to be a probability measure on $M$, Egorov's theorem gives that this pointwise convergence is almost uniform. We may therefore assume (after potentially moving to a subsequence) that the sequences $\{Kf_n - \lambda f_n\}_{n = 1}^\infty$, $\{Kg_n-\eta g_n\}_{n = 1}^\infty$ converge to $0$ pointwise in $x \in M$ and furthermore are such that for all $\varepsilon > 0$ there exists $\mu$-measurable sets $E_1,E_2 \subset M$ with $\mu(E_1),\mu(E_2) < \varepsilon$ so that $Kf_n - \lambda f_n$ converges to 0 uniformly on $M\setminus E_1$ and $Kg_n - \eta g_n$ converges uniformly to zero on $M\setminus E_2$.   

Let us now define the bounded functions $f_{n,m,k}$ and $g_{n,m,k}$, derived from $f_n$ and $g_n$, by
\begin{equation}\label{Fns}
    \begin{split}
	f_{n,m,k}(x) &= \begin{cases}
		m & f_n(x) \geq m \\
		f_n(x) & \frac{1}{k} < f_n(x) < m \\
		\frac{1}{k} & 0 \leq f_n(x) \leq \frac{1}{k} \\
		-\frac{1}{k} & -\frac{1}{k} \leq f_n(x) < 0 \\
		f_n(x) & -m < f_n(x) < -\frac{1}{k} \\
		-m & f_n(x) \leq -m \\
	\end{cases} \\ 
	g_{n,m,k}(x) &= \begin{cases}
		m & g_n(x) \geq m \\
		g_n(x) & \frac{1}{k} < g_n(x) < m \\
		\frac{1}{k} & 0 \leq g_n(x) \leq \frac{1}{k} \\
		-\frac{1}{k} & -\frac{1}{k} \leq g_n(x) < 0 \\
		g_n(x) & -m < g_n(x) < -\frac{1}{k} \\
		-m & g_n(x) \leq -m \\
	\end{cases}
    \end{split}
\end{equation}
for each $x \in M$ and $n,m,k \geq 1$. For each $1 \leq m < \infty$ we have that $|f_{n,m,k}(x)g_{n,m,k}(x)| \leq m^2$, thus giving that the pointwise product $f_{n,m,k}g_{n,m,k}$ belongs to $L^2$ since $\mu(M) = 1$. Furthermore, $|f_{n,m,k}(x)g_{n,m,k}(x)| \geq \frac{1}{k^2}$, giving that $\|f_{n,m,k}g_{n,m,k}\|_2 \geq \frac{1}{k^2} > 0$, for each $k \geq 1$, again coming from the fact that $\mu$ is a probability measure on $M$.

We will show that the nonzero unit vectors $f_{n,m,k}g_{n,m,k}/\|f_{n,m,k}g_{n,m,k}\|_2 \in L^2$ can be used to create a sequence that satisfies Weyl's criterion for a sequence of appropriately chosen $(n,m,k) \in\mathbb{N}^3$. First, we have that
\begin{equation}\label{WeylSeq}
	\begin{split}
		\frac{1}{\|f_{n,m,k}g_{n,m,k}\|_2} &\|Kf_{n,m,k}g_{n,m,k} - \lambda\eta f_{n,m,k}g_{n,m,k}\|_2 \\
		 &= \frac{1}{\|f_{n,m,k}g_{n,m,k}\|_2} \|f_{n,m,k}\circ T\cdot g_{n,m,k}\circ T - \lambda\eta f_{n,m,k}g_{n,m,k}\|_2 \\
		 &\leq \frac{1}{\|f_{n,m,k}g_{n,m,k}\|_2} \|(f_{n,m,k}\circ T - \lambda f_{n,m,k}) \cdot g_{n,m,k}\circ T \|_2 \\
		 & \quad + \frac{1}{\|f_{n,m,k}g_{n,m,k}\|_2} \|(g_{n,m,k}\circ T - \eta g_{n,m,k}) \cdot \lambda f_{n,m,k} \|_2 \\ 
		 &\leq k^2m\|f_{n,m,k}\circ T - \lambda f_{n,m,k}\|_2 + |\lambda|k^2m\|g_{n,m,k}\circ T - \eta g_{n,m,k}\|_2
	\end{split}
\end{equation}  
where $\circ$ denotes function composition and we have used the fact that $|f_{n,m,k}(x)|,|g_{n,m,k}\circ T (x)| \leq m$ for all $x \in M$. We now proceed to bound $\|f_{n,m,k}\circ T - \lambda f_{n,m,k}\|_2$ and $\|g_{n,m,k}\circ T - \lambda g_{n,m,k}\|_2$ appropriately. Both terms can be bounded in exactly the same way and therefore we will only provide the full details for one.

Let us focus on bounding $\|f_{n,m,k}\circ T - \lambda f_{n,m,k}\|_2$ in full detail. First, from the almost uniform pointwise convergence of the sequence $\{Kf_n - \lambda f_n\}_{n = 1}^\infty$ to zero, it follows that for each $m \geq 1$ there exists $\mu$-measurable sets $M_m \subseteq M$ such that $(Kf_n - \lambda f_n)$ converges uniformly on $M_m$ and 
\begin{equation}
	\mu(M\setminus M_m) \leq \frac{1}{4m^6}.
\end{equation} 
Furthermore, from the uniform convergence on $M_m$, for each $k \geq 1$ there exists $N_{f,k} \geq 1$ such that 
\begin{equation}\label{Uniformf}
	|Kf_n(x) - \lambda f_n(x)| = |f_n(Tx) - \lambda f_n(x)| < \frac{1}{k^6}
\end{equation}
for all $x \in M_m$ and $n \geq N_{f,k}$. For any set $A \subset M$ we will adopt the shorthand 
\begin{equation}
	\|f\|_A := \bigg(\int_A |f(x)|^2 \drm \mu\bigg)^\frac{1}{2}
\end{equation}
to denote the $L^2$ norm of the function $f$ restricted to any $\mu$-measurable set $A \subseteq M$. Then, combining the above with the fact that $|f_{n,m,k}(Tx) - \lambda f_{n,m,k}(x)| \leq 2m$ for all $x \in M\setminus M_m$ we have that 
\begin{equation}\label{fBnd}
	\begin{split}
		\|f_{n,m,k}\circ T - \lambda f_{n,m,k}\|_2 &\leq \|f_{n,m,k}\circ T - \lambda f_{n,m,k}\|_{M_m} + \|f_{n,m,k}\circ T - \lambda f_{n,m,k}\|_{M\setminus M_m} \\	
		&\leq \|f_{n,m,k}\circ T - \lambda f_{n,m,k}\|_{M_m} + 2m\sqrt{\mu(M\setminus M_m)} \\
		&\leq \|f_{n,m,k}\circ T - \lambda f_{n,m,k}\|_{M_m}  + \frac{1}{m^2},
	\end{split}
\end{equation}
for all $n,m,k \geq 1$. Using \eqref{Uniformf} we find that on $M_m$ we have
\begin{equation}
	|f_{n,m,k}(Tx) - \lambda f_{n,m,k}(x)| \leq |f_n(Tx) - \lambda f_n(x)| < \frac{1}{k^6}
\end{equation}
for each $n \geq N_{f,k}$, and so
\begin{equation}
	\|f_{n,m,k}\circ T - \lambda f_{n,m,k}\|_{M_m} < \frac{1}{k^6}\sqrt{\mu(M_m)} \leq \frac{1}{k^6}
\end{equation}
since $M_m \subseteq M$, giving that $\mu(M_m) \leq 1$. Combining the above with \eqref{fBnd} gives that
\begin{equation}\label{fBnd2}
	k^2m\|f_{n,m,k}\circ T - \lambda f_{n,m,k}\|_2 \leq \frac{m}{k^4} + \frac{k^2}{m}	
\end{equation}
for all $n \geq N_{f,k}$ and $k,m \geq 1$. 

We may repeat the above for the sequence $\{Kg_n - \eta g_n\}_{n = 1}^\infty$ to find that for each $k \geq 1$ there exists $N_{g,k} \geq 1$ such that 
\begin{equation}\label{gBnd2}
	|\lambda| k^2m\|f_{n,m,k}\circ T - \lambda f_{n,m,k}\|_2 \leq \frac{|\lambda|m}{k^4} + \frac{|\lambda|k^2}{m},	
\end{equation}	
for all $n \geq N_{g,k}$ and $k,m \geq 1$. Putting the bounds \eqref{fBnd2} and \eqref{gBnd2} into \eqref{WeylSeq} gives that  
\begin{equation}\label{WeylBnd}
	\frac{1}{\|f_{n,m,k}g_{n,m,k}\|_2} \|Kf_{n,m,k}g_{n,m,k} - \lambda\eta f_{n,m,k}g_{n,m,k}\|_2 \leq (1 + |\lambda|)\bigg(\frac{m}{k^4} + \frac{k^2}{m}\bigg).
\end{equation}
Let us define $n_k = \max\{N_{f,k},N_{g,k}\}$ and define 
\begin{equation}
	\psi_k := \frac{1}{\|f_{n_k,k^3,k}g_{n_k,k^3,k}\|_2}f_{n_k,k^3,k}g_{n_k,k^3,k} \in L^2 
\end{equation}
for each $k \geq 1$. That is, we have taken $n = n_k$ and $m = k^3$ for each $k \geq 1$. Clearly $\|\psi_k\|_2 = 1$ for all $k \geq 1$, and from \eqref{WeylBnd} we have that 
\begin{equation}
	\begin{split}
		\|K\psi_k - \lambda\eta\psi_k\|_2 &= \frac{1}{\|f_{n_k,k^3,k}g_{n_k,k^3,k}\|_2} \|Kf_{n_k,k^3,k}g_{n_k,k^3,k} - \lambda\eta f_{n_k,k^3,k}g_{n_k,k^3,k}\|_2 \\ 
		&\leq (1 + |\lambda|)\bigg(\frac{k^3}{k^4} + \frac{k^2}{k^3}\bigg)  \\
		&\leq \frac{2(1 + |\lambda|)}{k}.
	\end{split}
\end{equation}
Therefore, taking $k \to \infty$ gives that $\|K\psi_k - \lambda\eta\psi_k\|_2 \to 0$, showing that $\psi_k$ is a sequence satisfying Weyl's criterion, and so $\lambda\eta \in \sigma(K)$. This concludes the proof.

\section*{Acknowledgements}

\noindent This work was partially supported by an NSERC Discovery Grant and the Fondes de Recherche du Qu\'ebec – Nature et Technologies (FRQNT).

\bibliographystyle{amsplain}
\bibliography{references.bib}

\end{document}